\documentclass[11pt]{article}

\usepackage[T1]{fontenc}
\usepackage{lmodern}

\usepackage{amsmath,amssymb,amsthm,mathtools}
\usepackage[margin=1in]{geometry}

\usepackage{hyperref}
\hypersetup{
 colorlinks=true,
 linkcolor=blue,
 citecolor=blue,
 urlcolor=blue
}
\usepackage{authblk}
\newtheorem{theorem}{Theorem}
\newtheorem{lemma}[theorem]{Lemma}

\title{Uniqueness of RS Saddle Point for Ising Perceptron}

\author{Shuta Nakajima}
\affil{Department of Mathematics, Keio University}
\date{}

\begin{document}
\maketitle

\begin{abstract}
We study the replica-symmetric saddle point equations for the Ising perceptron with Gaussian disorder and margin $\kappa\ge 0$. We prove that for each $\kappa\ge 0$ there is a critical capacity
$\alpha_c(\kappa)=\frac{2}{\pi\,\mathbb E[(\kappa-Z)_+^2]}$,
where $Z$ is a standard normal and $(x)_+=\max\{x,0\}$, such that the saddle point equation has a unique solution for $\alpha\in(0,\alpha_c(\kappa))$ and has no solution when $\alpha\ge \alpha_c(\kappa)$.
When $\alpha\uparrow \alpha_c(\kappa)$ and $\kappa>0$, the replica-symmetric free energy at this solution  diverges to $-\infty$.
In the zero-margin case $\kappa=0$, Ding and Sun obtained a conditional uniqueness result, with one step verified numerically. Our argument gives a fully analytic proof without computer assistance. We used GPT-5 to help develop intermediate proof steps and to perform sanity-check computations.
\end{abstract}

\section{Setup and main results}

Let $\phi$ and $\Phi$ be the standard normal density and cumulative distribution function, respectively:
\[
\phi(u)=\frac{1}{\sqrt{2\pi}}e^{-u^2/2},
\qquad
\Phi(u)=\int_{-\infty}^u \phi(s)\,ds,
\qquad
\overline{\Phi}(u)=1-\Phi(u)=\int_u^{\infty}\phi(s)\,ds.
\]
Define the inverse Mills ratio by 
\[
E(u):=\frac{\phi(u)}{\overline{\Phi}(u)}\qquad (u\in\mathbb R),
\]
and for $q\in[0,1)$ set
\[
F_q(x):=\frac{1}{\sqrt{1-q}}\,E\!\left(\frac{\kappa-x}{\sqrt{1-q}}\right),\qquad x\in\mathbb R.
\]
Let $Z\sim\mathcal N(0,1)$. For $r\ge 0$, define
\[
P(r):=\mathbb E\big[\tanh^2(\sqrt{r}\,Z)\big]\in[0,1),
\]
and for $\alpha>0,\,\kappa\geq 0$ and $q\in[0,1)$ define
\[
R_{\kappa}(q,\alpha):=\alpha\,\mathbb E\big[F_q(\sqrt{q}\,Z)^2\big]\in(0,\infty).
\]
We study solutions $(q,r)\in[0,1)\times[0,\infty)$ to
\begin{equation}\label{eq:system}
\left\{
\begin{aligned}
q&=P(r),\\
r&=R_{\kappa}(q,\alpha).
\end{aligned}
\right.
\end{equation}

It is convenient to introduce
\begin{equation}\label{eq:defB}
B(q):=(1-q)\,\mathbb E\left[E\!\left(\frac{\kappa-\sqrt{q}\,Z}{\sqrt{1-q}}\right)^2\right],\qquad q\in[0,1),
\end{equation}
so that the second equation in \eqref{eq:system} can be rewritten as
\begin{equation}\label{eq:r_in_terms_of_B}
r=\frac{\alpha}{(1-q)^2}\,B(q).
\end{equation}
Set
\[
C_{\kappa}:=\mathbb E\big[(\kappa-Z)_+^2\big]=(\kappa^2+1)\Phi(\kappa)+\kappa\phi(\kappa),
\qquad
\alpha_c(\kappa):=\frac{2}{\pi C_{\kappa}}.
\]

\begin{theorem}\label{thm:main}
Fix $\kappa\geq 0$ and $\alpha\in (0,\alpha_c(\kappa))$. There exists a unique solution to the equation \eqref{eq:system}. Moreover, if $\alpha\ge \alpha_c(\kappa)$, then there is no solution.
\end{theorem}
For $\alpha<\alpha_c(\kappa)$, let $(q_\alpha,r_\alpha)$ be the unique solution to the equation \eqref{eq:system} given by Theorem~\ref{thm:main}.
\begin{theorem}\label{thm:2ndmain}
Fix $\kappa\geq 0$. When $\alpha\uparrow\alpha_c(\kappa)$, \[
q_\alpha\to 1
\ \text{ and }\ 
r_\alpha\to\infty.
\]
\end{theorem}
\begin{theorem}\label{thm: bound for threshold}
Fix $\kappa\geq 0$. 
 Define the RS functional
\begin{equation}\label{eq:RSfunctional}
\mathcal F_\kappa(\alpha;q,r)
:=
-\frac{r(1-q)}{2}
+\mathbb E\big[\log(2\cosh(\sqrt r\,Z))\big]
+\alpha\,\mathbb E\!\left[\log \overline\Phi\!\left(\frac{\kappa-\sqrt q\,Z}{\sqrt{1-q}}\right)\right].
\end{equation}
For $\alpha\in(0,\alpha_c(\kappa))$, define the Gardner formula (also known as the replica-symmetric formula) by $$\mathrm{RS}_\star(\alpha,\kappa):=\mathcal F_\kappa(\alpha;q_\alpha,r_\alpha),$$
where $(q_\alpha,r_\alpha)$ is the unique solution given by Theorem~\ref{thm:main}. Then
\[
\lim_{\alpha\uparrow\alpha_c(\kappa)} \mathrm{RS}_\star(\alpha,\kappa)=-\infty.
\]
\end{theorem}
\section{Motivation and related research}

\subsection{The Ising perceptron and the storage problem}

The binary (Ising) perceptron is a standard high-dimensional model for memory storage in a single-layer neural network. It can also be seen as a random constraint satisfaction problem with many constraints.
Fix $N\ge 1$ and let $g^1,\dots,g^M\in\mathbb R^N$ be i.i.d.\ random patterns. In this paper, we focus on Gaussian disorder,
\(
g^m \sim \mathcal N(0,I_N)\quad\text{i.i.d.}
\) 
We set the number of constraints $M=\lfloor \alpha N\rfloor$ with $\alpha>0$ fixed as $N\to\infty$.
For a margin parameter $\kappa\in\mathbb R$, define the feasible set
\[
\mathcal S_{N,M}^{(\kappa)}
:=
\Bigl\{
\sigma\in\{-1,1\}^N:
\frac{\langle g^m,\sigma\rangle}{\sqrt N}\ge \kappa
\ \text{for all } m\in[M]
\Bigr\},
\qquad [M]:=\{1,\dots,M\},
\]
where $\langle\cdot,\cdot\rangle$ is the Euclidean inner product in $\mathbb R^N$, and the corresponding number of solutions
\[
Z_{N,M}^{(\kappa)}
:=
\bigl|\mathcal S_{N,M}^{(\kappa)}\bigr|.
\]
When $\kappa=0$, this becomes the classical half-space intersection model. In this case, $Z_{N,M}^{(0)}$ counts vertices of the hypercube in the intersection of $M$ random half-spaces.
We study the quenched free energy 
\[
F_N(\alpha,\kappa)
:=
\frac{1}{N} \log Z_{N,\lfloor \alpha N\rfloor}^{(\kappa)},
\]
and we ask when $Z_{N,\lfloor \alpha N\rfloor}^{(\kappa)}>0$.
It is expected that feasibility exhibits a sharp transition in $\alpha$. The critical capacity was predicted by statistical physics in the late 1980s; see \cite{GardnerDerrida1988,KrauthMezard1989,Mezard1989}.

\subsection{Replica-symmetric free energy and the saddle point system}

A main goal is to determine the limiting free energy $\lim_{N\to\infty}F_N(\alpha,\kappa)$ and relate it to the feasibility threshold.
Methods from statistical physics predict that this limit is given by an explicit variational formula, often called the Gardner or Krauth--M\'ezard prediction for the Ising perceptron \cite{KrauthMezard1989,Mezard1989}.
Talagrand proved the replica-symmetric formula at small densities $\alpha$, for the half-space model and for more general activation functions \cite{Talagrand2000,Talagrand2011}.
Later, Bolthausen, Sun, Xu, and Nakajima gave another proof based on the conditional second-moment method via approximate message passing (AMP), which applies to a wider class of activation functions \cite{BolthausenNakajimaSunXu2022}.

To describe the previous work more precisely, we use the generalized replica-symmetric expression from \cite{BolthausenNakajimaSunXu2022}.
Let $U:\mathbb R\to[0,1]$ be a measurable activation function.
For $M=\lfloor \alpha N\rfloor$ and $G\in\mathbb R^{M\times N}$ with i.i.d.\ standard Gaussian entries (rows $g^m$), define
\[
Z_N
:=
\sum_{\sigma\in\{-1,1\}^N}
\prod_{m=1}^M U\Bigl(\frac{\langle g^m,\sigma\rangle}{\sqrt N}\Bigr).
\]
The half-space model corresponds to $U(x)=\mathbf 1\{x\ge \kappa\}$. In that case $Z_N=Z_{N,M}^{(\kappa)}$.
The replica-symmetric functional uses a pair $(q,r)$. We view $q$ as an overlap parameter and $r$ as a variance parameter.
Let $Z,\xi\sim\mathcal N(0,1)$ be independent.
For $q\in(0,1)$ define
\[
L_q(x):=\log\mathbb E\Bigl[U\bigl(x+\sqrt{1-q}\,\xi\bigr)\Bigr],
\qquad
F_q(x):=\frac{d}{dx}L_q(x).
\]
The replica-symmetric free energy is the functional
\begin{equation}\label{eq:RS-functional}
\mathrm{RS}(\alpha;U;q,r)
:=
-\frac{r(1-q)}{2}
+
\mathbb E\Bigl[\log\bigl(2\cosh(\sqrt r\,Z)\bigr)\Bigr]
+
\alpha\,\mathbb E\Bigl[L_q(\sqrt q\,Z)\Bigr].
\end{equation}

The following theorem summarizes the replica-symmetric formula and the uniqueness of the associated fixed point equation at sufficiently small $\alpha$.

\begin{theorem}[\cite{Talagrand2000}; \cite{BolthausenNakajimaSunXu2022}]
Assume that $U$ satisfies the following conditions (cf. \cite{BolthausenNakajimaSunXu2022}).
\begin{itemize}
\item[\textup{(A1)}] $U:\mathbb R\to[0,1]$ is measurable, and for $\xi\sim\mathcal N(0,1)$,
\[
\mathbb E_{\xi}\bigl[\xi\,U(\xi)\bigr]
=
\int_{\mathbb R} z\,U(z)\,\phi(z)\,dz
\neq 0.
\]
\item[\textup{(A2)}] Let $\xi,\xi'\sim\mathcal N(0,1)$ be i.i.d.\ Assume that
\[
(K_2)_1(U)
:=
\sup\Biggl\{
\frac{\mathbb E_{\xi,\xi'}\bigl[(\xi-\xi')^2\,U(x+c\xi)\,U(x+c\xi')\bigr]}
 {\mathbb E_{\xi,\xi'}\bigl[U(x+c\xi)\,U(x+c\xi')\bigr]}
:\ x\in\mathbb R,\ \frac{2}{5}\le c\le \frac{7}{3}
\Biggr\}<\infty,
\]
with the convention that the ratio is $+\infty$ when the denominator is $0$. 
\end{itemize}

Then there exists $\alpha_0=\alpha_0(U)>0$ such that for every $\alpha\in(0,\alpha_0]$ the fixed point equation \eqref{eq:RS-fixed-point} has a unique solution $(q_\star,r_\star)$, and
\[
\frac{1}{N}\log Z_N \xrightarrow[N\to\infty]{\mathbb P} \mathrm{RS}(\alpha;U;q_\star,r_\star).
\]
In the special cases where $U(x)=\mathbf 1\{x\ge \kappa\}$, or $U$ is sufficiently smooth, the convergence above was proved earlier by Talagrand \cite{Talagrand2000,Talagrand2011}.
\end{theorem}

The saddle point equations are the stationarity conditions for \eqref{eq:RS-functional}. They can be written as the fixed point equation
\begin{equation}\label{eq:RS-fixed-point}
q
=
\mathbb E\bigl[\tanh^2(\sqrt r\,Z)\bigr],
\qquad
r
=
\alpha\,\mathbb E\bigl[F_q(\sqrt q\,Z)^2\bigr].
\end{equation}
In the half-space case $U(x)=\mathbf 1\{x\ge \kappa\}$, we have the explicit formulas
\[
L_q(x)=\log \Phi\Bigl(\frac{x-\kappa}{\sqrt{1-q}}\Bigr),
\qquad
F_q(x)
=
\frac{1}{\sqrt{1-q}}\,
\frac{\phi\bigl(\frac{x-\kappa}{\sqrt{1-q}}\bigr)}{\Phi\bigl(\frac{x-\kappa}{\sqrt{1-q}}\bigr)},
\]
where $\Phi$ and $\phi$ are the standard Gaussian distribution function and density. In this case \eqref{eq:RS-fixed-point} is equivalent to \eqref{eq:system} above.

There are two ways to understand the fixed point equation \eqref{eq:RS-fixed-point}.
One way is from replica computations: the moment calculation of $Z_N$ becomes a saddle point problem over order parameters, and \eqref{eq:RS-fixed-point} gives the self-consistency condition.
Another way is from the TAP/AMP description of the Gibbs measure of $Z_N$. In this view, $q$ and $r$ describe macroscopic second moments of the AMP iterations, and \eqref{eq:RS-fixed-point} follows from the AMP state evolution \cite{DingSun2025,BolthausenNakajimaSunXu2022}.

\subsection{Sharp thresholds and capacity results}

There is also work on sharp thresholds for feasibility.
For Bernoulli disorder versions of the half-space model, Xu proved a sharp threshold phenomenon from $\{Z_{N,M}^{(\kappa)}>0\}$ to $\{Z_{N,M}^{(\kappa)}=0\}$ of the half-space model \cite{Xu2021} in the sense of \eqref{eq:sharp-seq-exp}.
For general disorder (including Gaussian disorder) and more general $\{0,1\}$-valued activation functions, Sun and Nakajima proved self-averaging of the free energy, the existence of a sharp threshold sequence, and universality with respect to the disorder \cite{NakajimaSun2022}.
These results suggest a clear picture: feasibility transitions occur in a narrow window as $\alpha$ varies, though the window may depend on $N$ in those work.

Another line of work aims to determine the critical capacity of the Gaussian half-space model in the zero-margin case $\kappa=0$.
Ding and Sun proved, conditional on an additional analytic hypothesis, a lower bound matching the constant $\alpha_\star\approx 0.833$ predicted by Krauth and M\'ezard, using a second-moment argument together with an AMP-motivated truncation \cite{DingSun2025}.
More recently, Huang obtained a conditionally matching upper bound under a related maximization assumption, yielding a conditional verification of the Krauth--M\'ezard prediction \cite{Huang2024}.
We summarize the streamlined version of their results here.
\begin{theorem}[\cite{DingSun2025}; \cite{Xu2021}; \cite{NakajimaSun2022}; \cite{Huang2024}]\label{thm:conditional_RS_threshold_kappa0}
Let $\kappa=0$. Then there exists $\alpha_\star(0)\in(0,\infty)$ such that
\[
\mathrm{RS}_\star(\alpha,0)>0 \quad \text{for } \alpha<\alpha_\star(0),
\qquad
\mathrm{RS}_\star(\alpha,0)<0 \quad \text{for } \alpha>\alpha_\star(0).
\]
Under the analytic hypotheses of \cite{DingSun2025,Huang2024}, it holds that 
\[
\lim_{N\to\infty}\mathbb P\bigl(Z^{(0)}_{N,\lfloor \alpha N\rfloor}>0\bigr)
=
\begin{cases}
1, & \alpha<\alpha_\star(0),\\
0, & \alpha>\alpha_\star(0).
\end{cases}
\]
\end{theorem}

Motivated by this, it is natural to view the replica-symmetric prediction as defining a candidate feasibility threshold via the point at which the predicted limiting free energy vanishes. These works show that understanding the replica-symmetric equations, and in particular the structure of the saddle point system, is important for rigorous capacity bounds.

\subsection{Main result: uniqueness of the replica-symmetric saddle point} 

The replica-symmetric formula is most useful in parameter regimes where the saddle point equation \eqref{eq:RS-fixed-point} has a unique solution $(q,r)$. Uniqueness matters for several reasons. It eliminates ambiguity in selecting a branch of solutions, yields regular dependence of $(q,r)$ on the parameters $(\alpha,\kappa)$, and controls the derivatives of the replica-symmetric free energy \eqref{eq:RS-functional}. 

Related work of Ding and Sun \cite{DingSun2025} established the uniqueness when $\kappa=0$, with one step verified via a computer-assisted numerical check.

In this paper we give a fully analytic proof of uniqueness for the half-space Ising perceptron. Concretely, when $U(x)=\mathbf{1}\{x\ge \kappa\}$ with $\kappa\ge 0$, we show that the map in \eqref{eq:RS-fixed-point} has at most one fixed point in the relevant domain. As a consequence, the replica-symmetric free energy can be written unambiguously as
\[
\mathrm{RS}(\alpha;U)
=
\mathrm{RS}(\alpha;U;q(\alpha),r(\alpha)),
\]
where $(q(\alpha),r(\alpha))$ is the unique solution to \eqref{eq:RS-fixed-point}.

\subsection{Connection to the multi-label classification}

The Ising perceptron is closely related to the multi-label classification with random labels. 
Consider random inputs $g^1,\dots,g^M\in\mathbb R^N$ and $L$-dimensional outputs $y^m\in\{-1,1\}^L$. Assume that the entries of $g^m$ are i.i.d., and that the labels $(y^m_\ell)$ are i.i.d.\ uniform on $\{-1,1\}$ and independent of the inputs $(g^m)$. We write $[n]:=[1,n]\cap \mathbb Z$. 
A binary weight matrix $W=(w_{\ell i})\in\{-1,1\}^{L\times N}$ satisfies the multi-output classification constraints if
\[\operatorname{sgn}\bigl(W g^m\bigr)=y^m,\quad\text{for all } m\in[M],
\]
 which is written as
 \[
\operatorname{sgn}\bigl(\langle w^\ell,g^m\rangle\bigr)=y^m_\ell
\quad\text{for all } m\in[M],\ \ell\in[L],\]
where $w^\ell$ is the $\ell$-th row of $W$. Equivalently, since $\langle w^\ell,g^m\rangle\neq 0$ almost surely,
\[
y^m_\ell\,\langle w^\ell,g^m\rangle \ge 0
\quad\text{for all } m,\ell.
\]
If the input distribution is symmetric, meaning $g^m\stackrel d= -g^m$, then the label-flipped patterns
$\tilde g^{m,\ell}:=y^m_\ell g^m$ satisfy $\tilde g^{m,\ell}\stackrel d=g^m$ for each fixed $\ell$. Therefore, for each $\ell$ the feasibility event for row $\ell$ has the same distribution as $\{Z_{N,M}^{(0)}>0\}$.
Let $\mathcal E$ be the event that there is \emph{no} matrix solution $W$. Then
\[
\mathcal E=\bigcup_{\ell=1}^L\Bigl\{\text{no feasible } w^\ell \in\{-1,1\}^N \text{ for output }\ell\Bigr\}.
\]
By a union bound, we obtain the bounds
\begin{equation}\label{eq:union-bound}
\mathbb P(Z_{N,M}^{(0)}=0)
\le
\mathbb P(\mathcal E)
\le
L\,\mathbb P(Z_{N,M}^{(0)}=0).
\end{equation}
The same reduction applies with a positive margin: if one requires $y^m_\ell \langle w^\ell,g^m\rangle/\sqrt N\ge \kappa$ for all $m\in[M]$ and $\ell\in[L]$, then the analogue of \eqref{eq:union-bound} holds with $Z_{N,M}^{(\kappa)}$.

We use the sharp-threshold sequence proved in \cite{Xu2021, NakajimaSun2022} to summarize the conclusion.
Fix $\kappa\ge 0$, and write $M=\lfloor \alpha N\rfloor$.
Assume there exists a sequence $\alpha_N=\alpha_N(\kappa)$ such that for every fixed $\varepsilon>0$ there is $c_\varepsilon>0$ with, for all large $N$,
\begin{equation}\label{eq:sharp-seq-exp}
\begin{split}
 & \alpha\le \alpha_N-\varepsilon
\ \Longrightarrow\
\mathbb P\!\left(Z_{N,\lfloor \alpha N\rfloor}^{(\kappa)}=0\right)\le e^{-c_\varepsilon N},\\
&\alpha\ge \alpha_N+\varepsilon
\ \Longrightarrow\
\mathbb P\!\left(Z_{N,\lfloor \alpha N\rfloor}^{(\kappa)}=0\right)\ge 1-e^{-c_\varepsilon N}.
\end{split}
\end{equation}
% Combining \eqref{eq:sharp-seq-exp} with the union bound \eqref{eq:union-bound} shows that the multi-output infeasibility event
% $\mathcal E$ has the same qualitative behavior outside a narrow window around $\alpha_N$.

If $\alpha\le \alpha_N-\varepsilon$, then
\[
\mathbb P(\mathcal E)\le L\,\mathbb P\!\left(Z_{N,\lfloor \alpha N\rfloor}^{(\kappa)}=0\right)\le L\,e^{-c_\varepsilon N},
\]
so whenever $\log L=o(N)$ (e.g., fixed $L$ or polynomial $L$), the probability $\mathbb P(\mathcal E)$ still decays exponentially in $N$. If $\alpha\ge \alpha_N+\varepsilon$, then
\[
\mathbb P(\mathcal E)\ge \mathbb P\!\left(Z_{N,\lfloor \alpha N\rfloor}^{(\kappa)}=0\right)\ge 1-e^{-c_\varepsilon N},
\]
which converges to $1$ as $N\to \infty$.

Thus, except for a narrow window, the two models behave in the same way.

\subsection{How we use GPT-5 in this project}
We used GPT-5 as a tool for computation and reasoning throughout the project (in particular, Lemma~\ref{lem:g-decreasing-positive}, \ref{lem:g-uniform-negative-TOBECHECKED} below). In a typical workflow, we supplied a problem together with a proposed proof strategy, and asked GPT-5 to suggest an outline and to elaborate intermediate steps. Because the outputs contained gaps or errors, we independently verified the details and either corrected them ourselves or asked GPT-5 to address specific issues. We repeated this process iteratively until reaching the final version of each proof. All arguments in the manuscript are fully analytic and were not computer-verified. Overall, we found generative AI helpful for converting informal sketches or computer-assisted arguments into fully rigorous analytic proofs. Numerical computer-assisted proofs typically depend on approximating a large number of quantities, while an analytic proof often reduces the problem to a small set of key estimates. This type of reduction appears to match well with the strengths of generative AI.

\section{Auxiliary results for Theorem~\ref{thm:main}}
\subsection{Roadmap of the proof}
The aim of this section is to reduce the fixed point equation \eqref{eq:system} to a one-dimensional equation and to isolate the main technical ingredient needed for Theorem~\ref{thm:main} (the strict monotonicity of $B$, proved in Section~3).

Recall that, with $B$ defined in \eqref{eq:defB}, the second equation in \eqref{eq:system} can be rewritten as \eqref{eq:r_in_terms_of_B}:
\[
r=\frac{\alpha}{(1-q)^2}\,B(q).
\]
Substituting $q=P(r)$ yields
\[
r\bigl(1-P(r)\bigr)^2=\alpha\,B\bigl(P(r)\bigr).
\]
Define
\begin{equation}\label{eq:defA}
A(r):=r\bigl(1-P(r)\bigr)^2
=r\Bigl(\mathbb E\big[\operatorname{sech}^2(\sqrt{r}\,Z)\big]\Bigr)^2,\qquad r\ge 0.
\end{equation}
Letting $q=P(r)$,
 solving \eqref{eq:system} is equivalent to solving the one-dimensional equation
\begin{equation}\label{eq:one-dimensional}
A(r)=\alpha\,B\bigl(P(r)\bigr),\qquad r\ge 0,
\end{equation}

The remaining ingredients are standard: $P$ is strictly increasing (Lemma~\ref{lem:P_properties-TOBECHECKED}); $A$ is strictly increasing with range $[0,2/\pi)$ (Lemma~\ref{lem:A-TOBECHECKED}); and $B$ is finite/continuous on $[0,1)$ with endpoint limits $B(0)=E(\kappa)^2$ and $\lim_{q\uparrow 1}B(q)=C_\kappa$ (Lemmas~\ref{lem:B_cont0-TOBECHECKED} and~\ref{lem:B_endpoints-TOBECHECKED}). Once $B$ is known to be strictly decreasing, the difference $A(r)-\alpha B(P(r))$ is strictly increasing and has at most one root, and comparing endpoint limits yields existence/nonexistence with threshold $\alpha_c(\kappa)=2/(\pi C_\kappa)$.

\subsection{Properties of $P$}
\begin{lemma}\label{lem:P_properties-TOBECHECKED}
The function $P:[0,\infty)\to[0,1)$ defined by
\[
P(r)=\mathbb E\big[\tanh^2(\sqrt r\,Z)\big]
\]
is continuous and strictly increasing, satisfies $P(0)=0$ and $P(r)<1$ for all finite $r$, and
\[
\lim_{r\to\infty}P(r)=1.
\]
\end{lemma}
\begin{proof}
For each fixed $z\in\mathbb R$, the map $r\mapsto \tanh^2(\sqrt r\,z)$ is nondecreasing and continuous on $[0,\infty)$. Since $0\le \tanh^2(\sqrt r\,Z)\le 1$, dominated convergence yields continuity of $P$ and preserves monotonicity.

To see strictness, fix $0\le r_1<r_2$. Then $\tanh^2(\sqrt{r_1}\,z)<\tanh^2(\sqrt{r_2}\,z)$ for every $z\neq 0$, and $\mathbb{P}(Z=0)=0$, hence $P(r_1)<P(r_2)$.

The identity $P(0)=0$ is immediate. If $r<\infty$, then $\tanh^2(\sqrt r\,z)<1$ for every $z\in\mathbb R$, hence $P(r)<1$. Finally, when $r\to\infty$ one has $\tanh^2(\sqrt r\,z)\to 1$ for every $z\ne 0$, so dominated convergence yields $P(r)\to 1$.
\end{proof}

\subsection{The function $A$: monotonicity and range}
\begin{lemma}\label{lem:A-TOBECHECKED}
The function $A:[0,\infty)\to[0,\infty)$ defined in \eqref{eq:defA} is continuous on $[0,\infty)$ and strictly increasing on $(0,\infty)$. Moreover,
\[
A(0)=0,
\qquad
\lim_{r\to\infty}A(r)=\frac{2}{\pi}.
\]
In particular, $A$ is a bijection from $[0,\infty)$ onto $[0,2/\pi)$.
\end{lemma}

\begin{proof}
For $r>0$, let
\[
S(r):=\mathbb E\big[\operatorname{sech}^2(\sqrt{r}\,Z)\big]
=\int_{\mathbb R}\operatorname{sech}^2(\sqrt{r}\,z)\,\phi(z)\,dz.
\]
Then $A(r)=rS(r)^2$. Make the change of variables $y=\sqrt{r}\,z$ to obtain
\[
S(r)=\frac{1}{\sqrt{2\pi r}}\int_{\mathbb R}\operatorname{sech}^2(y)\,e^{-y^2/(2r)}\,dy.
\]
Hence
\begin{equation}\label{eq:A_as_I}
A(r)=\frac{1}{2\pi}\,I(r)^2,
\qquad
I(r):=\int_{\mathbb R}\operatorname{sech}^2(y)\,e^{-y^2/(2r)}\,dy.
\end{equation}
For $r>0$ we may differentiate under the integral sign:
\[
I'(r)=\int_{\mathbb R}\operatorname{sech}^2(y)\,e^{-y^2/(2r)}\,\frac{y^2}{2r^2}\,dy>0.
\]
Therefore $I$ is strictly increasing on $(0,\infty)$, hence so is $A$ by \eqref{eq:A_as_I}. Continuity of $A$ follows from continuity of $P$ (Lemma~\ref{lem:P_properties-TOBECHECKED}) and the representation $A(r)=r(1-P(r))^2$.

Finally, for each $y$ one has $e^{-y^2/(2r)}\uparrow 1$ as $r\to\infty$, so by monotone convergence,
\[
\lim_{r\to\infty}I(r)=\int_{\mathbb R}\operatorname{sech}^2(y)\,dy
=\bigl[\tanh(y)\bigr]_{-\infty}^{\infty}=2.
\]
Using \eqref{eq:A_as_I} yields $\lim_{r\to\infty}A(r)=\frac{1}{2\pi}\cdot 2^2=\frac{2}{\pi}$.
\end{proof}

\subsection{The function $B$: finiteness, continuity, and endpoint values}
In this subsection, we record the basic analytic properties of $B$; the strict monotonicity of $B$ is deferred to Section~3.

The following lemma provides a simple bound for the inverse Mills ratio. Since the proof is elementary, we defer it to the appendix.
\begin{lemma}\label{lem:Mills_bound-TOBECHECKED}
Recall $E(u)=\phi(u)/\overline{\Phi}(u)$. There exists a finite constant $C>0$ such that for all
$u\in\mathbb R$,
\[
0<E(u)=\frac{\phi(u)}{\overline{\Phi}(u)} \le u_+ + C,\qquad u_+:=\max\{u,0\}.
\]
Moreover, for all $u>0$,
\[
u<\frac{\phi(u)}{\overline{\Phi}(u)}\le u+\frac{1}{u}.
\]
\end{lemma}

\begin{lemma}\label{lem:B_cont0-TOBECHECKED}
The function
\[
B(q)=(1-q)\,\mathbb E\!\left[E\!\left(\frac{\kappa-\sqrt{q}\,Z}{\sqrt{1-q}}\right)^2\right]
\]
is finite and continuous for all $q\in[0,1)$.
\end{lemma}
\begin{proof}
Fix $q_0\in[0,1)$ and let $q_n\in [0,1)$ with $q_n\to q_0$. Define
\[
U_n:=\frac{\kappa-\sqrt{q_n}\,Z}{\sqrt{1-q_n}},\qquad
U_0:=\frac{\kappa-\sqrt{q_0}\,Z}{\sqrt{1-q_0}}.
\]
Then $U_n\to U_0$ almost surely, hence $(1-q_n)E(U_n)^2\to (1-q_0)E(U_0)^2$ almost surely.

To apply dominated convergence, use Lemma~\ref{lem:Mills_bound-TOBECHECKED}:
\[
E(U_n)^2\le (U_n^+ + C)^2\le 2\big((U_n^+)^2+C^2\big)\le 2(U_n^2+C^2).
\]
Multiplying by $(1-q_n)$ gives
\[
(1-q_n)E(U_n)^2
\le 2\big((1-q_n)U_n^2 + C^2(1-q_n)\big)
=2\big((\kappa-\sqrt{q_n}\,Z)^2 + C^2(1-q_n)\big).
\]
Since $0\le q_n\le 1$ eventually, $(\kappa-\sqrt{q_n}\,Z)^2\le 2(\kappa^2+Z^2)$ and
$(1-q_n)\le 1$, so
\[
(1-q_n)E(U_n)^2\le 4(\kappa^2+Z^2)+2C^2.
\]
The right-hand side is integrable because $\mathbb E[Z^2]=1$. Therefore dominated convergence
applies, and we obtain $B(q_n)\to B(q_0)$. Finiteness follows from the same domination.
\end{proof}

\begin{lemma}\label{lem:B_endpoints-TOBECHECKED}
Let $B$ be defined by \eqref{eq:defB}. Then
\[
B(0)=E(\kappa)^2,
\qquad
\lim_{q\uparrow 1}B(q)=C_{\kappa}=\mathbb E\big[(\kappa-Z)_+^2\big]=(\kappa^2+1)\Phi(\kappa)+\kappa\phi(\kappa).
\]
\end{lemma}

\begin{proof}
At $q=0$, it is trivial to see $B(0)=(1-0)E(\kappa)^2$. For the limit $q\uparrow 1$, set
\[
U_q:=\frac{\kappa-\sqrt{q}\,Z}{\sqrt{1-q}}.
\]
For each fixed $z\in\mathbb R$,
\[
(1-q)\,E\bigl(U_q(z)\bigr)^2\to (\kappa-z)_+^2
\qquad (q\uparrow 1),
\]
because $U_q(z)\to+\infty$ when $z<\kappa$, $U_q(z)\to-\infty$ when $z>\kappa$,  $E(u)\sim u$ as $u\to+\infty$ by Lemma~\ref{lem:Mills_bound-TOBECHECKED} and $E(u) \to 0$ as $u\to-\infty$ since $|E(u)| \leq 2\phi(u)$ for $u<0$.

Moreover, one may dominate the integrand as follows. By Lemma~\ref{lem:Mills_bound-TOBECHECKED}, there is a constant
$C>0$ such that $E(u)\le u_+ + C$ for all $u\in\mathbb{R}$. Hence, by $(a+b)^2\leq 2(a^2+b^2)$, for any
$q\in(0,1)$ and $z\in\mathbb{R}$,
\begin{align*}
 (1-q)E(U_q(z))^2&\le 2(1-q)(U_q(z))^2+2C^2(1-q)\\
&=2(\kappa-\sqrt{q}\,z)^2+2C^2(1-q)\\
&\le 2(\kappa+|z|)^2+2C^2.
\end{align*}
The right-hand side is integrable against $\phi(z)\,dz$, so dominated convergence yields
\[
\lim_{q\uparrow 1}B(q)=\int_{\mathbb R}(\kappa-z)_+^2\,\phi(z)\,dz=\mathbb E\big[(\kappa-Z)_+^2\big]=:C_{\kappa}.
\]
Finally,
\[
C_{\kappa}=\int_{-\infty}^{\kappa}(\kappa-z)^2\,\phi(z)\,dz
=\kappa^2\Phi(\kappa)-2\kappa\int_{-\infty}^{\kappa}z\phi(z)\,dz+\int_{-\infty}^{\kappa}z^2\phi(z)\,dz.
\]
Using $\phi'(z)=-z\phi(z)$ gives $\int_{-\infty}^{\kappa}z\phi(z)\,dz=-\phi(\kappa)$ and $\int_{-\infty}^{\kappa}z^2\phi(z)\,dz=\Phi(\kappa)-\kappa\phi(\kappa)$, hence
\[
C_{\kappa}=(\kappa^2+1)\Phi(\kappa)+\kappa\phi(\kappa).
\]
\end{proof}

Recall that $\alpha_c(\kappa)=2/(\pi C_\kappa)$; this is the value at which the limiting upper range $2/\pi$ of $A$ matches the limiting size $\alpha\,C_\kappa$ of the right-hand side in \eqref{eq:one-dimensional}.

\subsection{Reduction of Theorem~\ref{thm:main} to monotonicity of $B$} 
\begin{lemma}\label{lem:main_reduction-TOBECHECKED}
Assume that $B$ is strictly decreasing on $[0,1)$. Fix $\alpha>0$ and define
\[
f(r):=A(r)-\alpha\,B\bigl(P(r)\bigr),\qquad r\ge 0.
\]
Then $f$ is strictly increasing on $(0,\infty)$. In particular, the one-dimensional equation \eqref{eq:one-dimensional} has at most one solution $r\in[0,\infty)$, and the fixed point equation \eqref{eq:system} admits at most one solution $(q,r)\in[0,1)\times[0,\infty)$.
\end{lemma}
\begin{proof}
By Lemma~\ref{lem:A-TOBECHECKED}, $A$ is strictly increasing on $(0,\infty)$. By Lemma~\ref{lem:P_properties-TOBECHECKED} and the assumed strict monotonicity of $B$, the composition $r\mapsto B(P(r))$ is strictly decreasing on $(0,\infty)$. Hence $f$ is strictly increasing on $(0,\infty)$.

Moreover, by Lemmas~\ref{lem:P_properties-TOBECHECKED},\,\ref{lem:A-TOBECHECKED}, and \ref{lem:B_endpoints-TOBECHECKED},
\[
f(0)=A(0)-\alpha B(P(0))=-\alpha B(0)=-\alpha E(\kappa)^2<0,
\]
so any zero of $f$ must lie in $(0,\infty)$. Since $f$ is strictly increasing on $(0,\infty)$, it has at most one zero. Therefore \eqref{eq:one-dimensional}, which is equivalent to $f(r)=0$, has at most one solution $r\in[0,\infty)$. The correspondence $q=P(r)$ then yields at most one solution $(q,r)$ to \eqref{eq:system}.
\end{proof}

Thus, the proof of Theorem~\ref{thm:main} reduces to the following monotonicity property of $B$.
\begin{theorem}[Monotonicity of $B$]\label{thm:B_monotone}
The function $q\mapsto B(q)$ is strictly decreasing on $[0,1)$.
\end{theorem}
Section~\ref{section: monotone B} is devoted to the proof of Theorem~\ref{thm:B_monotone}.

\section{Monotonicity of $B$}\label{section: monotone B}
Recall that $Z\sim N(0,1)$ and for $u\in\mathbb{R}$,
\[
\phi(u) \coloneqq \frac{1}{\sqrt{2\pi}}e^{-u^{2}/2},\quad
\Phi(u)\coloneqq \int_{-\infty}^{u}\phi(s)\,ds,\quad
\overline{\Phi}(u)\coloneqq 1-\Phi(u),\quad 
E(u)\coloneqq \frac{\phi(u)}{\overline{\Phi}(u)},
\]
and for $q\in[0,1)$ define
\[
U_q\coloneqq \frac{\kappa-\sqrt{q}\,Z}{\sqrt{1-q}},\qquad
B(q)\coloneqq (1-q)\,\mathbb{E}\bigl[E(U_q)^2\bigr].
\]
We now prove Theorem~\ref{thm:B_monotone}.

\subsection{Preliminaries on $E$ and truncated Gaussians}

Let $X\sim N(0,1)$. For every $u\in\mathbb{R}$,
\[
\mathbb{P}(X\ge u)=\overline{\Phi}(u),\qquad
\mathbb{E}[X\mathbf{1}_{\{X\ge u\}}]=\phi(u),
\]
hence
\begin{equation}\label{eq:condmean}
E(u)=\frac{\phi(u)}{\overline{\Phi}(u)}=\mathbb{E}[X\mid X\ge u].
\end{equation}
Differentiating $E(u)=\phi(u)/\overline{\Phi}(u)$ and using $\phi'(u)=-u\phi(u)$ and
$\overline{\Phi}'(u)=-\phi(u)$ gives
\begin{equation}\label{eq:mprime}
E'(u)=E(u)^2-u\,E(u).
\end{equation}
Define also
\[
d(u)\coloneqq E(u)-u=\mathbb{E}[X-u\mid X\ge u].
\]
Let $\tilde{\mathbb{P}}(\,\cdot\,)\coloneqq \mathbb{P}(\,\cdot\,\mid X\ge u)$ and let $Y_u$ denote the random variable
\(
Y_u\coloneqq X-u
\)
under $\tilde{\mathbb{P}}$. Then
\(
d(u)=\tilde{\mathbb{E}}[Y_u].
\) 
Moreover, $Y_u\ge 0$ $\tilde{\mathbb{P}}$-a.s., and for every integer $k\ge 0$ its conditional moments are finite:
\[
\tilde{\mathbb{E}}\bigl[Y_u^k\bigr]
=\mathbb{E}\bigl[(X-u)^k\mid X\ge u\bigr]
=\frac{\mathbb{E}\bigl[(X-u)^k\mathbf{1}_{\{X\ge u\}}\bigr]}{\overline{\Phi}(u)}
\le \frac{\mathbb{E}\bigl[(|X|+|u|)^k\bigr]}{\overline{\Phi}(u)}<\infty,
\]
since $X\sim N(0,1)$ has moments of all orders and $\overline{\Phi}(u)>0$ for all $u\in\mathbb{R}$.

A useful identity is
\begin{equation}\label{eq:varidentity}
\begin{aligned}
1-E'(u)
&=1-\bigl(E(u)^2-uE(u)\bigr)\\
&=1+uE(u)-E(u)^2\\
&=\frac{\int_u^\infty x^2\phi(x)\,dx}{\overline\Phi(u)}
-\left(\frac{\int_u^\infty x\phi(x)\,dx}{\overline\Phi(u)}\right)^2\\
&=\mathbb{E}[X^{2}\mid X\ge u]-\bigl(\mathbb{E}[X\mid X\ge u]\bigr)^{2}
=\mathrm{Var}(X\mid X\ge u)>0.
\end{aligned}
\end{equation}
so $E'(u)<1$ for all $u$, and consequently
\begin{equation}\label{eq:dprime}
d'(u)=E'(u)-1=-\mathrm{Var}(X\mid X\ge u)<0.
\end{equation}
In particular, $d(u)>0$ and $d$ is strictly decreasing on $\mathbb{R}$.

\subsection{Derivative formula}
Fix $t\in(0,1)$ and let $Z\sim\mathcal N(0,1)$. Set
\begin{equation}\label{eq:U-def}
U_t\coloneqq \frac{\kappa-\sqrt{t}\,Z}{\sqrt{1-t}},
\end{equation}
so that $B(t)=(1-t)\,\mathbb E\bigl[E(U_t)^2\bigr].$ 
Let $h(u):=E(u)^2$. Then $B(t)=(1-t)\,\mathbb E[h(U_t)]$. We differentiate under the expectation,
\[
B'(t)=-\mathbb E[h(U_t)]+(1-t)\,\mathbb E\bigl[h'(U_t)\,\partial_t U_t\bigr].
\]
A direct computation gives
\[
\partial_t U_t=\frac{U_t}{2(1-t)}-\frac{Z}{2\sqrt t\,\sqrt{1-t}},
\]
hence
\begin{equation}\label{eq:Bprime_preIBP}
B'(t)=\mathbb E\Bigl[-h(U_t)+\frac12 U_t h'(U_t)\Bigr]
-\frac{\sqrt{1-t}}{2\sqrt t}\,\mathbb E\bigl[Z\,h'(U_t)\bigr].
\end{equation}

We now use Gaussian integration by parts: for smooth $\varphi$,
\[
\mathbb E[Z\varphi(Z)]=\mathbb E[\varphi'(Z)].
\]
Apply this with
\[
\varphi(z):=h'\!\left(\frac{\kappa-\sqrt t\,z}{\sqrt{1-t}}\right)=h'(U_t(z)).
\]
Since $\frac{d}{dz}U_t(z)=-\sqrt{\frac{t}{1-t}}$, we have
\[
\varphi'(z)=-\sqrt{\frac{t}{1-t}}\,h''(U_t(z)),
\]
and therefore
\[
\mathbb E\bigl[Z\,h'(U_t)\bigr]
=-\sqrt{\frac{t}{1-t}}\,\mathbb E\bigl[h''(U_t)\bigr].
\]
Substituting into \eqref{eq:Bprime_preIBP} yields
\[
B'(t)=\mathbb E\Bigl[-h(U_t)+\frac12 U_t h'(U_t)+\frac12 h''(U_t)\Bigr].
\]

Next,
\[
h'(u)=2E(u)E'(u),\qquad
h''(u)=2(E'(u))^2+2E(u)E''(u).
\]
Thus, we have
\[
B'(t)= \mathbb{E}\Big[-(E(U_t))^2+U_t E(U_t)E'(U_t)+(E'(U_t))^2+E(U_t)E''(U_t)\Big].
\]
Using the identity (recall $E'(u)=E(u)^2-u\,E(u)$ from \eqref{eq:mprime})
\[
E''(u)=2E(u)E'(u)-E(u)-uE'(u),
\]
we get
\[
uE(u)E'(u)+E(u)E''(u)-(E(u))^2=2E(u)^2\bigl(E'(u)-1\bigr).
\]
Hence
\begin{equation}\label{eq:Bprime-TOBECHECKED}
    B'(t)=\mathbb E\Bigl[(E'(U_t))^2+2E(U_t)^2\bigl(E'(U_t)-1\bigr)\Bigr]
=\mathbb E\bigl[g(U_t)\bigr],
\end{equation}
 where we define, for $u\in\mathbb{R}$,
\begin{equation}\label{eq:g-def}
g(u)\coloneqq E'(u)^2-2\bigl(1-E'(u)\bigr)E(u)^2.
\end{equation}

Therefore, to prove that $B$ is strictly decreasing on $(0,1)$, it suffices to show that
\begin{equation}\label{eq:goal}
\mathbb{E}\bigl[g(U_t)\bigr]<0\qquad\text{for all }t\in(0,1).
\end{equation}

\subsection{The function $g$ on $[0,\infty)$}
Recall that $E(u)=\frac{\phi(u)}{\overline{\Phi}(u)}=\mathbb{E}[X\mid X\ge u]$ with standard normal $X$. Using \eqref{eq:mprime}, i.e., $E'(u)=E(u)^2-u\,E(u)$, one can rewrite $g$ as the quartic form
\begin{equation}\label{eq:g-expanded}
g(u)=E(u)^2\bigl(3E(u)^2-4u\,E(u)+u^2-2\bigr).
\end{equation}
We will also use the first few conditional moments of $Y_u$: for $k\geq 0$, 
\[
\mu_k(u):= \tilde{\mathbb{E}}\bigl[(Y_u)^k\bigr]
=\mathbb{E}\bigl[(X-u)^k\mid X\ge u\bigr].
\]
Recall that $d(u)\coloneqq E(u)-u=\mu_1(u)$. A standard integration by parts recursion gives
\begin{align}
\mu_0(u)&=1,\quad 
\mu_1(u)=d(u),\label{eq:moments-TOBECHECKED}\\
\mu_2(u)&=1-u\,d(u),\nonumber\\
\mu_3(u)&=(u^2+2)\,d(u)-u,\nonumber\\
\mu_4(u)&=u^2+3-u(u^2+5)\,d(u).\nonumber
\end{align}

\begin{lemma}\label{lem:g-decreasing-positive}
The function $g$ is strictly decreasing on $[0,\infty)$. In particular, $g(u)\le g(0)$ for all $u\ge 0$.
Moreover,
\[
g(0)=\frac{12}{\pi^2}-\frac{4}{\pi}=-\frac{4(\pi-3)}{\pi^2}.
\]
\end{lemma}
Before going into the proof of the lemma above, we introduce a simple analytic lemma which will be proved in the appendix.
\begin{lemma}\label{lem:F negative-TOBECHECKED}
Let $(x,y)\in\mathbb R^2$ satisfy $x\ge 0$, $0<y<\frac{2}{3}$, and
\begin{equation}\label{eq:consFneg_lemma}
x+2y\ge 1,\qquad
x^2+xy-3x-3y+2\ge 0,\qquad
x+y<1.
\end{equation}
Then it holds that 
\begin{equation}\label{eq:FdefFneg_lemma-TOBECHECKED}
F(x,y):=x^2+6xy+6y^2-x-4y<0.
\end{equation}
\end{lemma}

\begin{proof}[Proof of Lemma~\ref{lem:g-decreasing-positive} assuming Lemma~\ref{lem:F negative-TOBECHECKED}]
We differentiate $g$ and rewrite it in a form that is convenient for moment inequalities. Set $E=E(u)$, $E'=E'(u)=E(u)^2-u\,E(u)$,
and $d=d(u)=E(u)-u$. A direct differentiation of \eqref{eq:g-expanded}, using \eqref{eq:mprime} and
$d'=E'-1$, yields the identity
\begin{equation}\label{eq:gprime}
g'(u)=2\,E(u)^2\,H(u),
\qquad
H(u)\coloneqq u^2 d + 6u d^2 + 6 d^3 - u - 4d .
\end{equation}
Since $E(u)>0$ for all $u$, the sign of $g'(u)$ is the sign of $H(u)$.

Fix $u\ge 0$. Let $X\sim N(0,1)$ and let $Y_u=X-u$ under the conditional law given $\{X\ge u\}$.
Then $Y_u\ge 0$ almost surely, and $\mu_k(u)=\tilde{\mathbb{E}}[Y_u^k]$ are given in \eqref{eq:moments-TOBECHECKED}.

For any polynomial $P$ we have $\tilde{\mathbb{E}}[P(Y_u)^2]\ge 0$, so the (unshifted) Hankel moment matrices
$\bigl(\mu_{i+j}(u)\bigr)_{i,j\ge 0}$ are positive semidefinite. Moreover, since $Y_u\ge 0$,
also $\tilde{\mathbb{E}}\!\bigl[Y_u\,P(Y_u)^2\bigr]\ge 0$, so the \emph{shifted} Hankel matrices
$\bigl(\mu_{i+j+1}(u)\bigr)_{i,j\ge 0}$ are positive semidefinite. In particular,
\[
M_1(u)\coloneqq
\begin{pmatrix}
\mu_1 & \mu_2\\
\mu_2 & \mu_3
\end{pmatrix}
\succeq 0,
\qquad
\widetilde M_2(u)\coloneqq
\begin{pmatrix}
\mu_0 & \mu_1 & \mu_2\\
\mu_1 & \mu_2 & \mu_3\\
\mu_2 & \mu_3 & \mu_4
\end{pmatrix}
\succeq 0,
\]
and $\mathrm{Var}(Y_u)=\mu_2-\mu_1^2>0$.

Hence $\det M_1(u)\ge 0$ and, using \eqref{eq:moments-TOBECHECKED},
\begin{equation}\label{eq:detM1-TOBECHECKED}
\det M_1(u)=\mu_1 \mu_3-\mu_2^2 = u d +2 d^2 -1 \ge 0.
\end{equation}
Similarly, $\det \widetilde M_2(u)\ge 0$ yields
\begin{equation}\label{eq:detM2-TOBECHECKED}
\begin{split}
    \det \widetilde M_2(u) &=\mu_0(\mu_2\mu_4-\mu_3^2)-\mu_1(\mu_1\mu_4-\mu_2\mu_3)+\mu_2(\mu_1\mu_3-\mu_2^2)\\
    &= u^2 d^2 + u d^3 -3 d^2 -3 u d +2 \ge 0.
    \end{split}
\end{equation}
Finally,
\begin{equation}\label{eq:varY-TOBECHECKED}
\mathrm{Var}(Y_u)=\mu_2-\mu_1^2=1-u d-d^2>0,
\end{equation}
so $u d + d^2 < 1$.

Introduce the variables
\[
x\coloneqq u d,\qquad y\coloneqq d^2.
\]
Since $d=\mathbb{E}[X-u\mid X\ge u]>0$, for $u\ge 0$ one has $x\ge 0$ and $y>0$. The inequalities \eqref{eq:detM1-TOBECHECKED}, \eqref{eq:detM2-TOBECHECKED},
\eqref{eq:varY-TOBECHECKED} become
\begin{equation}\label{eq:constraints-TOBECHECKED}
x+2y\ge 1,\qquad
x^2+xy-3x-3y+2\ge 0,\qquad
x+y<1.
\end{equation}
Also, $u\mapsto d(u)$ is decreasing by \eqref{eq:dprime}, so $d(u)\le d(0)=E(0)=\sqrt{2/\pi}$ and hence
\begin{equation}\label{eq:ybound-TOBECHECKED}
0<y\le \frac{2}{\pi}<\frac{2}{3}.
\end{equation}

Now rewrite $H(u)$ in terms of $(x,y)$: 
\begin{equation}\label{eq:Hxy-TOBECHECKED}
 F(x,y):= d\,H(u)=u^2 d^2 + 6u d^3 + 6 d^4 - ud - 4d^2= x^2+6xy+6y^2-x-4y.
\end{equation}
By Lemma~\ref{lem:F negative-TOBECHECKED} with $d>0$, we have $H(u)<0$ for all $u\ge 0$. By \eqref{eq:gprime},
$g'(u)<0$ on $[0,\infty)$, so $g$ is strictly decreasing there.

Finally, at $u=0$ one has $E(0)=\phi(0)/\overline{\Phi}(0)=\sqrt{2/\pi}$ and by \eqref{eq:mprime}
$E'(0)=E(0)^2=2/\pi$. Plugging into \eqref{eq:g-def} gives
\begin{equation}\label{eq: g-zero-value-TOBECHECKED}
    g(0)=\left(\frac{2}{\pi}\right)^2-2\left(1-\frac{2}{\pi}\right)\left(\frac{2}{\pi}\right)
=\frac{12}{\pi^2}-\frac{4}{\pi}
=-\frac{4(\pi-3)}{\pi^2}.
\end{equation}

\end{proof}

\subsection{$g$ on $(-\infty,0]$ via critical points}

\begin{lemma}\label{lem:g-uniform-negative-TOBECHECKED}

For all $u\le 0$ one has
\[
g(u)\le \frac{1}{18}.
\]
\end{lemma}

\begin{proof}
First note that $g$ is continuous on $\mathbb{R}$ (because $E$ is smooth on $\mathbb{R}$),
and $g(u)\to 0$ as $u\to -\infty$ since $E(u)=\phi(u)/\overline{\Phi}(u)\le 2\phi(u)$ for $u\le 0$
and $\phi(u)\to 0$ as $u\to-\infty$. Moreover, $|u|\phi(u)\to 0$, so $|u|E(u)\le 2|u|\phi(u)\to 0$ and
\eqref{eq:mprime} gives $E'(u)=E(u)^2-uE(u)\to 0$, hence $g(u)\to 0$.
Also, $g(0)<0$ by Lemma \ref{lem:g-decreasing-positive}.

If $\sup_{u\le 0}g(u)=0$, then $g(u)\le 0\le 1/18$ for all $u\le 0$ and there is nothing to prove.
Otherwise, $\sup_{u\le 0}g(u)>0$. Since $g(0)<0$, this supremum cannot be attained at $u=0$ and therefore
must be attained at some interior maximizer $u_\star<0$.

At such an interior maximizer, $g'(u_\star)=0$. By \eqref{eq:gprime}, this implies $H(u_\star)=0$,
where $H$ is as in \eqref{eq:gprime}. Write  $d_\star:=d(u_\star)=E(u_\star)-u_\star$, and define
\[
r_\star\coloneqq -\frac{u_\star}{d_\star}=-\frac{u_\star}{E(u_\star)-u_\star}.
\]
Since $u_\star<0$ and $d_\star>0$, one has $r_\star>0$. Moreover, 
$$0<\phi(u_\star)/\overline{\Phi}(u_\star)=E(u_\star)=u_\star+d_\star=d_\star(1-r_\star),$$ so $r_\star<1$.
Hence $r_\star\in(0,1)$.

The equation $H(u_\star)=0$ becomes, after substituting $u_\star=-r_\star d_\star$,
\[
u_\star^2 d_\star + 6u_\star d_\star^2 + 6 d_\star^3 -4d_\star -u_\star
= (r_\star^2-6r_\star+6)d_\star^3 + (r_\star-4)d_\star = 0.
\]
Dividing by $d_\star>0$ gives
\[
(r_\star^2-6r_\star+6)d_\star^2 + (r_\star-4)=0,
\qquad\text{so}\qquad
d_\star^2=\frac{4-r_\star}{r_\star^2-6r_\star+6}.
\]
Now evaluate $g(u_\star)$ using \eqref{eq:g-expanded}, i.e., $g(u_\star)=E(u_\star)^2\bigl(3E(u_\star)^2-4u_\star\,E(u_\star)+u_\star^2-2\bigr)
$. With $E(u_\star)=u_\star+d_\star=d_\star(1-r_\star)$ and
$u_\star d_\star=-r_\star d_\star^2$, one has
\[
g(u_\star)=d_\star^2(1-r_\star)^2\bigl((-2r_\star+3)d_\star^2-2\bigr).
\]
Substituting $d_\star^2=(4-r_\star)/(r_\star^2-6r_\star+6)$, a cancellation occurs:
\[
((-2r_\star+3)d_\star^2-2)
=\frac{(3-2r_\star)(4-r_\star)-2(r_\star^2-6r_\star+6)}{r_\star^2-6r_\star+6}
=\frac{r_\star}{r_\star^2-6r_\star+6}.
\]
Therefore the value of $g$ at any critical point $u=u_\star<0$ is
\begin{equation}\label{eq:g-critical-r-TOBECHECKED}
g(u_\star)=\frac{r_\star(4-r_\star)(1-r_\star)^2}{(r_\star^2-6r_\star+6)^2},
\qquad r_\star\in(0,1).
\end{equation}

It remains to bound the right-hand side:
\[
\frac{r(4-r)(1-r)^2}{(r^2-6r+6)^2}\le \frac{1}{18},
\qquad r\in(0,1).
\]
Since this is a straightforward estimate, we defer the proof to the appendix; see Lemma~\ref{lem: rational function bound-TOBECHECKED}.

Combining this with \eqref{eq:g-critical-r-TOBECHECKED},  every interior maximizer $u_\star<0$
satisfies $g(u_\star)\le 1/18$. Since $g(u)\le \max\{0,g(u_\star)\}$ for all $u\le 0$ and $0\le 1/18$,
it follows that $g(u)\le 1/18$ for all $u\le 0$.
\end{proof}

\subsection{Completion of the proof: monotonicity of $B$}

\begin{proof}[Proof of Theorem~\ref{thm:B_monotone}]
Fix $t\in(0,1)$ and write $U_t=\frac{\kappa-\sqrt{t}\,Z}{\sqrt{1-t}}$ as in \eqref{eq:U-def}. Since  $\kappa/\sqrt{t}\geq 0$,
\[
\mathbb{P}(U_t<0)=\mathbb{P}\!\left(\frac{\kappa-\sqrt{t}\,Z}{\sqrt{1-t}}<0\right)
=\mathbb{P}(\sqrt{t}\,Z>\kappa)
=\overline{\Phi}\!\left(\frac{\kappa}{\sqrt{t}}\right)
\leq \frac{1}{2}.
\]

Split the expectation in \eqref{eq:Bprime-TOBECHECKED} according to the sign of $U_t$.
By Lemma \ref{lem:g-decreasing-positive}, for $U_t\ge 0$ one has $g(U_t)\le g(0)$.
By Lemma \ref{lem:g-uniform-negative-TOBECHECKED}, for $U_t<0$ one has $g(U_t)\le 1/18$.
Therefore, with $p_t\coloneqq \mathbb{P}(U_t<0)\in(0,1/2]$,
\begin{equation}\label{eq:split}
\mathbb{E}[g(U_t)]
\le g(0)\,(1-p_t) + \frac{1}{18}\,p_t.
\end{equation}

We now show $g(0)< -1/18$, so that the right-hand side of \eqref{eq:split} is strictly negative.
From Lemma \ref{lem:g-decreasing-positive}, we have $g(0)=-{4(\pi-3)}/{\pi^2}.$ Hence, it suffices to show that 
\begin{equation}\label{eq: pi-estimate-TOBECHECKED}
    -\frac{4(\pi-3)}{\pi^2}<-\frac{1}{18}.
\end{equation}

It is classical (Archimedes' bounds) that $\pi>\frac{223}{71}$. The function
\[
h(x)\coloneqq \frac{4(x-3)}{x^2}
\]
is strictly increasing for $x\in(3,6)$ because $h'(x)=\frac{4(6-x)}{x^3}>0$ there.
Since $\pi\in(3,6)$, we have
\[
\frac{4(\pi-3)}{\pi^2}=h(\pi)\ge h\!\left(\frac{223}{71}\right)
=\frac{4\bigl(\frac{223}{71}-3\bigr)}{\bigl(\frac{223}{71}\bigr)^2}
=\frac{201640}{3530759}
>\frac{1}{18}.
\]

Insert this into \eqref{eq:split}:
\[
\mathbb{E}[g(U_t)]
< g(0)\,(1-p_t) - g(0)\,p_t
= g(0)\,(1-2p_t)
\leq 0,
\]
since $g(0)<0$ and $p_t\leq 1/2$. By \eqref{eq:Bprime-TOBECHECKED}, $B'(t)=\mathbb{E}[g(U_t)]<0$ for all $t\in(0,1)$.

Therefore $B$ is strictly decreasing on $(0,1)$. Continuity of $B$ on $[0,1)$ follows directly
from the definition and dominated convergence, so $B$ is strictly decreasing on $[0,1)$ as claimed.
\end{proof}

\section{Proof of Theorem~\ref{thm:main},\,\ref{thm:2ndmain},\,\ref{thm: bound for threshold}}
\begin{proof}[Proof of Theorem \ref{thm:main}]
By Theorem~\ref{thm:B_monotone}, $B$ is strictly decreasing on $[0,1)$. Fix $\alpha>0$ and define $f(r)=A(r)-\alpha\,B\bigl(P(r)\bigr)$ as in Lemma~\ref{lem:main_reduction-TOBECHECKED}. Then $f$ is strictly increasing on $(0,\infty)$.

Moreover, by Lemmas~\ref{lem:A-TOBECHECKED}, \ref{lem:P_properties-TOBECHECKED}, and \ref{lem:B_cont0-TOBECHECKED}, the function $f$ is continuous on $[0,\infty)$. By Lemmas~\ref{lem:A-TOBECHECKED}, \ref{lem:P_properties-TOBECHECKED}, and \ref{lem:B_endpoints-TOBECHECKED},
\[
f(0)=A(0)-\alpha B(P(0))=-\alpha B(0)<0.
\]
Also, using Lemmas~\ref{lem:A-TOBECHECKED}, \ref{lem:P_properties-TOBECHECKED}, and \ref{lem:B_endpoints-TOBECHECKED},
\[
\lim_{r\to\infty}f(r)=\lim_{r\to\infty}A(r)-\alpha\lim_{r\to\infty}B(P(r))
=\frac{2}{\pi}-\alpha C_{\kappa}.
\]

If $\alpha<\alpha_c(\kappa)=\frac{2}{\pi C_{\kappa}}$, then $\lim_{r\to\infty}f(r)>0$, so by continuity there exists $r_*\in(0,\infty)$ with $f(r_*)=0$. Since $f$ is strictly increasing on $(0,\infty)$, this $r_*$ is unique. Using $R_\kappa(q,\alpha)=\alpha\,\mathbb E[F_q(\sqrt q Z)^2]=\alpha B(q)/(1-q)^2$, we get $r_*=R_\kappa(q_*,\alpha)$.
 Setting $q_*:=P(r_*)<1$, the identity $f(r_*)=0$ is equivalent to
\(
r_*=\alpha B(q_*)/(1-q_*)^2,
\)
so $(q_*,r_*)$ solves \eqref{eq:system}.

If $\alpha\ge \alpha_c(\kappa)$, then $\lim_{r\to\infty}f(r)\le 0$. Since $f(0)<0$ and $f$ is strictly increasing on $(0,\infty)$, it follows that $f(r)<0$ for all $r\ge 0$. Any solution $(q,r)$ of \eqref{eq:system} would satisfy $f(r)=0$ (since $q=P(r)$ and $r=R_\kappa(q,\alpha)=\alpha B(q)/(1-q)^2$), contradicting $f(r)<0$. Hence \eqref{eq:system} has no solution.

\end{proof}

\begin{proof}[{Proof of Theorem~\ref{thm:2ndmain}}]

Fix $\kappa\geq 0$. For each $\alpha\in(0,\alpha_c(\kappa))$, let $(q_\alpha,r_\alpha)$ be the unique
solution to \eqref{eq:system}. It suffices to show that for any sequence
$\alpha_n\uparrow\alpha_c(\kappa)$,
\[
r_{\alpha_n}\to\infty.
\]
Indeed, if $r_{\alpha_n}\to\infty$ then $q_{\alpha_n}=P(r_{\alpha_n})\to 1$ by Lemma
\ref{lem:P_properties-TOBECHECKED}. Since the choice of $\alpha_n$ is arbitrary, this yields
\[
r_\alpha\to\infty,\qquad q_\alpha\to 1\qquad\text{as }\alpha\uparrow\alpha_c(\kappa).
\]

Let $\alpha_n\uparrow \alpha_c(\kappa)$. Set $(q_n,r_n):=(q_{\alpha_n},r_{\alpha_n})$, so
\[
q_n=P(r_n),\qquad r_n=\frac{\alpha_n}{(1-q_n)^2}\,B(q_n).
\]

We will prove that $r_n\to\infty$. To this end, assume the contrary. Then there exists a subsequence, still denoted $(r_n)$, that is bounded.
By compactness, after passing to a further subsequence we may assume $r_n\to r_*\in[0,\infty)$.
By Lemma~\ref{lem:P_properties-TOBECHECKED} and $q_n=P(r_n)$,
\[
q_n\to q_*:=P(r_*).
\]
Since $r_*<\infty$, Lemma~\ref{lem:P_properties-TOBECHECKED} gives $q_*=P(r_*)<1$.

Now use the second equation and take limits. Because $q_n\to q_*<1$, we have
$(1-q_n)^{-2}\to(1-q_*)^{-2}$, and by Lemma~\ref{lem:B_cont0-TOBECHECKED},
$B(q_n)\to B(q_*)$. Also, $\alpha_n\to\alpha_c(\kappa)$. Therefore
\[
r_*=\lim_{n\to\infty} r_n
=\lim_{n\to\infty}\frac{\alpha_n}{(1-q_n)^2}\,B(q_n)
=\frac{\alpha_c(\kappa)}{(1-q_*)^2}\,B(q_*).
\]
Together with $q_*=P(r_*)$, this shows that $(q_*,r_*)$ solves \eqref{eq:system} at
$\alpha=\alpha_c(\kappa)$. This contradicts Theorem~\ref{thm:main}, which states that a solution
exists if and only if $\alpha<\alpha_c(\kappa)$. The contradiction proves $r_n\to\infty$.
\end{proof}
\begin{proof}[Proof of Theorem~\ref{thm: bound for threshold}]
Let $\alpha_n\uparrow \alpha_c(\kappa)$ and set $(q_n,r_n):=(q_{\alpha_n},r_{\alpha_n})$.
Write $\varepsilon_n:=1-q_n$. By Theorem~\ref{thm:2ndmain}, $\varepsilon_n\to 0$ and $r_n\to\infty$.
Recall also from Theorem~\ref{thm:main} that $(q_n,r_n)$ satisfies the fixed point relations
\[
q_n=\mathbb E\big[\tanh^2(\sqrt{r_n}\,Z)\big],
\qquad
r_n=\frac{\alpha_n}{\varepsilon_n^2}\,B(q_n).
\]
Define
\[
U_n:=\frac{\kappa-\sqrt{q_n}\,Z}{\sqrt{\varepsilon_n}},
\qquad
A_n:=\mathbb E\big[(\kappa-\sqrt{q_n}\,Z)_+^2\big].
\]
Then
\[
\mathrm{RS}_\star(\alpha_n,\kappa)
=
\log 2 -\frac{r_n\varepsilon_n}{2}
+\mathbb E\big[\log\cosh(\sqrt{r_n}\,Z)\big]
+\alpha_n\,\mathbb E\!\left[\log \overline{\Phi}(U_n)\right].
\]

For every $x\in\mathbb R$ one has $\log\cosh x\le x\tanh x$. Taking $x=\sqrt{r_n}\,Z$ and using Gaussian
integration by parts,
\begin{equation}\label{eq: cosh_sech_comparison-TOBECHECKED}
    \mathbb E\big[\log\cosh(\sqrt{r_n}\,Z)\big]
\le
\sqrt{r_n}\,\mathbb E\big[Z\tanh(\sqrt{r_n}\,Z)\big]
%=\sqrt{r_n}\,\mathbb E\big[(\tanh(\sqrt{r_n}\,Z))'\big]
=
r_n\,\mathbb E\big[\mathrm{sech}^2(\sqrt{r_n}\,Z)\big].
\end{equation}
Since $\mathrm{sech}^2=1-\tanh^2$ and $q_n=\mathbb E[\tanh^2(\sqrt{r_n}Z)]$, this gives
$\mathbb E[\mathrm{sech}^2(\sqrt{r_n}Z)]=1-q_n=\varepsilon_n$, hence
\[
-\frac{r_n\varepsilon_n}{2}+\mathbb E\big[\log\cosh(\sqrt{r_n}\,Z)\big]
\le \frac{r_n\varepsilon_n}{2}.
\]
Using $r_n=\frac{\alpha_n}{\varepsilon_n^2}B(q_n)$ yields
\begin{equation}\label{eq:spin_bd_unif_nostep_revised-TOBECHECKED}
-\frac{r_n\varepsilon_n}{2}+\mathbb E\big[\log\cosh(\sqrt{r_n}\,Z)\big]
\le
\frac{\alpha_n}{2\varepsilon_n}\,B(q_n).
\end{equation}

Fix $\delta\in(0,1)$. Since $q_n\to 1$, for all large $n$ we have $\sqrt{q_n}\ge 1/2$. On the event
$\{Z\le \kappa-\delta\}$,
\[
\kappa-\sqrt{q_n}Z\ge \kappa-\sqrt{q_n}(\kappa-\delta)
=\kappa(1-\sqrt{q_n})+\sqrt{q_n}\delta\ge \delta/2,
\]
so $U_n\ge \delta/(2\sqrt{\varepsilon_n})>0$ on $\{Z\le \kappa-\delta\}$.

For $u>0$, Chernoff's bound gives $$\overline{\Phi}(u)=\mathbb P(Z\geq u)\leq e^{-u^2}\mathbb E[e^{u Z}]= e^{-u^2/2}.$$ By Lemma~\ref{lem:Mills_bound-TOBECHECKED}, we have 
\begin{equation*}
 \begin{split}
 \overline{\Phi}(u)\le \phi(u)/u\le u^{-1}e^{-u^2/2}.
\end{split}
\end{equation*}
 Hence
\[
\log\overline{\Phi}(u)\le -\frac{u^2}{2},
\qquad
\log\overline{\Phi}(u)\le -\frac{u^2}{2}-\log u,
\qquad (u>0),
\]
and trivially $\log\overline{\Phi}(u)\le 0$ for all $u\in\mathbb R$.
Therefore, for all large $n$,
\[
\log\overline{\Phi}(U_n)
\le
-\frac{U_n^2}{2}\,\mathbf{1}_{\{U_n>0\}}
-(\log U_n)\,\mathbf{1}_{\{Z\le \kappa-\delta\}}.
\]
Taking expectations and using $U_n^2\,\mathbf{1}_{\{U_n>0\}}=\varepsilon_n^{-1}(\kappa-\sqrt{q_n}Z)_+^2$,
\[
\mathbb E\big[\log\overline{\Phi}(U_n)\big]
\le
-\frac{A_n}{2\varepsilon_n}
-\mathbb E\big[(\log U_n)\,\mathbf{1}_{\{Z\le \kappa-\delta\}}\big].
\]
On $\{Z\le \kappa-\delta\}$ we have $\log U_n=\log(\kappa-\sqrt{q_n}Z)-\frac12\log\varepsilon_n$, hence
\[
-\mathbb E\big[(\log U_n)\,\mathbf{1}_{\{Z\le \kappa-\delta\}}\big]
=
\frac{\Phi(\kappa-\delta)}{2}\log\varepsilon_n
-\mathbb E\big[\log(\kappa-\sqrt{q_n}Z)\,\mathbf{1}_{\{Z\le \kappa-\delta\}}\big].
\]
Since on $\{Z\le \kappa-\delta\}$ and for large $n$ we have $\kappa-\sqrt{q_n}Z\ge \delta/2$, it follows that
$\log(\kappa-\sqrt{q_n}Z)\ge \log(\delta/2)$ and therefore
\[
-\mathbb E\big[\log(\kappa-\sqrt{q_n}Z)\,\mathbf{1}_{\{Z\le \kappa-\delta\}}\big]\le -\log(\delta/2).
\]
Consequently, with $C(\delta):=-\log(\delta/2)\in(0,\infty)$,
\begin{equation}\label{eq:constraint_bd_unif_nostep_revised-TOBECHECKED}
\mathbb E\big[\log\overline{\Phi}(U_n)\big]
\le
-\frac{A_n}{2\varepsilon_n}
+\frac{\Phi(\kappa-\delta)}{2}\log\varepsilon_n
+C(\delta),
\end{equation}
for all large $n$.

Next, let $E(u):=\phi(u)/\overline{\Phi}(u)$ and recall 
\[
B(q_n)=\varepsilon_n\,\mathbb E\big[E(U_n)^2\big],
\qquad
A_n=\varepsilon_n\,\mathbb E\big[(U_n)_+^2\big].
\]
For $u>0$,  Lemma~\ref{lem:Mills_bound-TOBECHECKED}, $\overline{\Phi}(u)\le \phi(u)/u$, implies $E(u)\ge u$, hence
$E(u)^2\ge u^2=(u_+)^2$, while for $u\le 0$ one has $(u_+)^2=0\le E(u)^2$. Thus
$E(u)^2\ge (u_+)^2$ for all $u\in\mathbb R$, and so $B(q_n)\ge A_n$.

Moreover, there is a finite constant $C_0$ such that
\[
0\le E(u)^2-(u_+)^2\le C_0\qquad\text{for all }u\in\mathbb R.
\]
Indeed, if $u\le 0$, then $\overline{\Phi}(u)\ge 1/2$ and $\phi(u)\le \phi(0)$, so
$E(u)\le 2\phi(0)=\sqrt{2/\pi}$ and $E(u)^2-(u_+)^2=E(u)^2\le 2/\pi$.
If $u\ge 1$, the two-sided Mills bound $\overline{\Phi}(u)\ge \phi(u)\,u/(1+u^2)$ by Lemma~\ref{lem:Mills_bound-TOBECHECKED} gives
$E(u)\le (1+u^2)/u=u+1/u$, hence
\[
0\le E(u)^2-u^2\le (u+1/u)^2-u^2=2+1/u^2\le 3.
\]
Finally, on $u\in[0,1]$ the function $u\mapsto E(u)^2-u^2$ is continuous and hence bounded.
Taking $C_0$ as the maximum of these bounds yields the claim.

Therefore,
\begin{equation}\label{eq:BA_bd_unif_nostep_revised-TOBECHECKED}
0\le B(q_n)-A_n
=
\varepsilon_n\,\mathbb E\big[E(U_n)^2-(U_n)_+^2\big]
\le C_0\,\varepsilon_n.
\end{equation}

Combining \eqref{eq:spin_bd_unif_nostep_revised-TOBECHECKED},
\eqref{eq:constraint_bd_unif_nostep_revised-TOBECHECKED}, and \eqref{eq:BA_bd_unif_nostep_revised-TOBECHECKED},
\begin{align*}
\mathrm{RS}_\star(\alpha_n,\kappa)
&\le
\log 2+\frac{\alpha_n}{2\varepsilon_n}B(q_n)
+\alpha_n\left(
-\frac{A_n}{2\varepsilon_n}
+\frac{\Phi(\kappa-\delta)}{2}\log\varepsilon_n
+C(\delta)\right)\\
&=
\frac{\alpha_n\Phi(\kappa-\delta)}{2}\log\varepsilon_n
+\log 2
+\frac{\alpha_n}{2\varepsilon_n}\big(B(q_n)-A_n\big)
+\alpha_n C(\delta)\\
&\le
\frac{\alpha_n\Phi(\kappa-\delta)}{2}\log\varepsilon_n
+\log 2
+\frac{C_0\alpha_c(\kappa)}{2}
+\alpha_c(\kappa)\,C(\delta),
\end{align*}
for all large $n$, since $\alpha_n\le \alpha_c(\kappa)$ and $C(\delta)>0$.

Since $\varepsilon_n\to 0$, we have $\log\varepsilon_n\to -\infty$, and since
$\alpha_n\to\alpha_c(\kappa)>0$ and $\Phi(\kappa-\delta)>0$, the right-hand side tends to $-\infty$.
Therefore $\mathrm{RS}_\star(\alpha_n,\kappa)\to -\infty$, proving the claim.

\end{proof}

\appendix
\section{Appendix}
\subsection{Uniform growth bound for the inverse Mills ratio: Lemma~\ref{lem:Mills_bound-TOBECHECKED}}
Recall $E(u)=\phi(u)/\overline{\Phi}(u)$. We aim to prove that there exists $C>0$ such that for all
$u\in\mathbb R$,
\[
0<E(u)=\frac{\phi(u)}{\overline{\Phi}(u)} \le \max\{u,0\} + C.
\]
For all $u>0$,
\[
u<\frac{\phi(u)}{\overline{\Phi}(u)}\le u+\frac{1}{u}.
\]
\begin{proof}[Proof of Lemma~\ref{lem:Mills_bound-TOBECHECKED}]
For $u\le 1$, since $\overline{\Phi}$ is decreasing, we have $\overline{\Phi}(u)\ge \overline{\Phi}(1)>0$ and
$\phi(u)\le 1$, hence
\[
E(u)=\frac{\phi(u)}{\overline{\Phi}(u)}\le \frac{1}{\overline{\Phi}(1)}=:C_1<\infty.
\]
For $u>0$, integrate by parts to get the classical identity
\[
\overline{\Phi}(u)=\int_u^\infty \phi(t)\,dt
=\Big[-\frac{\phi(t)}{t}\Big]_{t=u}^\infty-\int_u^\infty\frac{\phi(t)}{t^2}\,dt
=\frac{\phi(u)}{u}-\int_u^\infty\frac{\phi(t)}{t^2}\,dt.
\]
Here the boundary term at $\infty$ vanishes since $\phi(t)/t\to 0$ as $t\to\infty$.
Since $t\ge u$ on $[u,\infty)$,
\[
\int_u^\infty \frac{\phi(t)}{t^2}\,dt\le \frac{1}{u^2}\int_u^\infty \phi(t)\,dt
=\frac{\overline{\Phi}(u)}{u^2},
\]
so
\begin{equation}\label{eq: two-sided mills}
 \overline{\Phi}(u)\ge \frac{\phi(u)}{u}-\frac{\overline{\Phi}(u)}{u^2},
\qquad\text{hence}\qquad
\frac{\phi(u)}{\overline{\Phi}(u)}\le u+\frac{1}{u}.
\end{equation}
In particular, for $u\ge 1$ we have $E(u)\le u+1$. Hence, we can choose $C:=\max\{C_1,1\}$.

Finally, for $u>0$,
\[ \overline{\Phi}(u)= \int_u^\infty \phi(t)\,dt=\int_u^\infty \frac{-\phi'(t)}{t}\,dt
=\Big[-\frac{\phi(t)}{t}\Big]_{u}^{\infty}-\int_u^\infty \frac{\phi(t)}{t^{2}}\,dt\\
< \frac{\phi(u)}u.\]

\end{proof}

\subsection{Proof of Lemma~\ref{lem:F negative-TOBECHECKED}}
We aim to prove the following: Let $(x,y)\in\mathbb R^2$ satisfy $x\ge 0$, $0<y<\frac{2}{3}$, and
\begin{equation}\label{eq:consFneg_lemma}
x+2y\ge 1,\qquad
x^2+xy-3x-3y+2\ge 0,\qquad
x+y<1.
\end{equation}
Then it holds that
\begin{equation}\label{eq:FdefFneg_lemma-TOBECHECKED}
F(x,y):=x^2+6xy+6y^2-x-4y<0.
\end{equation}
\begin{proof}
Fix $y\in(0,2/3)$.
Consider the quadratic constraint
\[
q_y(x):=x^2+xy-3x-3y+2\ge 0.
\]
Its roots are
\begin{equation}\label{eq:rootsFneg_lemma}
r_\pm(y)=\frac{3-y\pm\sqrt{y^2+6y+1}}{2}.
\end{equation}
Since
 ${y^2+6y+1}>(y+1)^2$ for $y>0$, we have
\[
r_+(y)=\frac{3-y+\sqrt{y^2+6y+1}}{2}>\frac{3-y+(y+1)}{2}=2>1.
\]
Because $q_y$ has leading coefficient $1$, the inequality $q_y(x)\ge 0$ implies
\[
x\in(-\infty,r_-(y)]\cup[r_+(y),\infty).
\]
From $x+y<1$ in \eqref{eq:consFneg_lemma} we get $x<1-y\le 1<r_+(y)$, so the second branch
is impossible and therefore
\(
x\le r_-(y).
\) 
Moreover, $r_-(y)+y<1$ holds for $y>0$, because it is equivalent to
$\sqrt{y^2+6y+1}>y+1$.

Combining $x\ge 0$ and $x+2y\ge 1$ from \eqref{eq:consFneg_lemma} with $x\le r_-(y)$, the feasible
set of $x$ for this fixed $y$ is the interval
\begin{equation}\label{eq:IyFneg_lemma}
I_y:=\Bigl[\max\{0,1-2y\},\,r_-(y)\Bigr].
\end{equation}
Note that if $I_y$ is empty, then the claim follows immediately for that $y$. Hence, we assume that $I_y$ is not empty.

For fixed $y$, the function $x\mapsto F(x,y)$ is a convex quadratic since
\[
F(x,y)=x^2+(6y-1)x+(6y^2-4y),\qquad \frac{\partial^2}{\partial x^2}F(x,y)=2>0.
\]
Hence the maximum of $F(\cdot,y)$ over the interval $I_y$ is attained at an endpoint since $I_y$ is a closed interval and 
 $F(\cdot,y)$ is convex and continuous.

For the left endpoint, there are two cases.
If $y\le \tfrac12$, then $\max\{0,1-2y\}=1-2y$ and
\begin{equation}\label{eq:leftEndpoint1Fneg_lemma}
F(1-2y,y)=-2y^2<0.
\end{equation}
If $y\ge \tfrac12$, then $\max\{0,1-2y\}=0$ and, using $y<\tfrac23$,
\begin{equation}\label{eq:leftEndpoint2Fneg_lemma}
F(0,y)=6y^2-4y=2y(3y-2)<0.
\end{equation}

For the right endpoint $x=r_-(y)$, use the identity $q_y(x)=0$ to eliminate $x^2$:
\[
F(x,y)=(x^2+xy-3x-3y+2) + (5xy+2x+6y^2-y-2)=5xy+2x+6y^2-y-2.
\]
Substituting $x=r_-(y)$ gives
\begin{equation}\label{eq:rightEndpointFneg_lemma}
F(r_-(y),y)=\frac{7y^2+11y+2-(5y+2)\sqrt{y^2+6y+1}}{2}.
\end{equation}
Both $7y^2+11y+2$ and $(5y+2)\sqrt{y^2+6y+1}$ are strictly positive for $y>0$,
so \eqref{eq:rightEndpointFneg_lemma} is negative if and only if
\[
(5y+2)\sqrt{y^2+6y+1}>7y^2+11y+2.
\]
Squaring and simplifying yields
\begin{equation}\label{eq:squareDiffFneg_lemma}
(5y+2)^2(y^2+6y+1)-(7y^2+11y+2)^2
=8y^3(2-3y),
\end{equation}
which is strictly positive for $y\in(0,2/3)$. Hence $F(r_-(y),y)<0$.

Therefore $F$ is negative at both endpoints of $I_y$, so $\max_{x\in I_y}F(x,y)<0$ and
$F(x,y)<0$ for every feasible $x$ at this $y$. Since $y$ was arbitrary in $(0,2/3)$, the conclusion follows.
\end{proof}
\subsection{A rational function inequality}
\begin{lemma}\label{lem: rational function bound-TOBECHECKED}
For any $r\in(0,1)$, 
\[ \frac{r(4-r)(1-r)^2}{(r^2-6r+6)^2}\le \frac{1}{18}.
\]
\end{lemma}
\begin{proof}
This is equivalent to
\[
(r^2-6r+6)^2 - 18\,r(4-r)(1-r)^2 \ge 0.
\]
Expanding gives the quartic polynomial
\[
P(r)\coloneqq 19r^4-120r^3+210r^2-144r+36.
\]
A useful factorization is
\[
P(r)-1=(r-1)\,C(r),
\qquad
C(r)\coloneqq 19r^3-101r^2+109r-35.
\]
On $[0,1]$, the cubic $C$ is strictly concave because
\[
C''(r)=114r-202<0\qquad\text{for all }r\in[0,1].
\]
The derivative $C'(r)=57r^2-202r+109$ is strictly decreasing, and $C'(0)=109>0$ while $C'(1)=-36<0$,
so there is a unique $r_0\in(0,1)$ with $C'(r_0)=0$, and $C$ attains its maximum at $r_0$.
Using $C'(r_0)=0$ (i.e.\ $57r_0^2=202r_0-109$) gives
\[
C(r_0)=\frac{5024-7976r_0}{171}.
\]
The quadratic formula yields
\[
r_0=\frac{202-\sqrt{202^2-4\cdot 57\cdot 109}}{2\cdot 57}
\]
and $202^2-4\cdot 57\cdot 109=15952<127^2$, hence $r_0>(202-127)/114=25/38$.
Therefore $7976r_0>7976\cdot(25/38)=99700/19>5024$, so $C(r_0)<0$.
Thus $C(r)<0$ for all $r\in[0,1]$. Since $r-1<0$ for $r\in(0,1)$, we obtain $(r-1)C(r)>0$ on $(0,1)$,
and thus $P(r)>1$ on $(0,1)$. In particular, $P(r)>0$ on $(0,1)$, proving the desired inequality.
\end{proof}

\end{document}